\documentclass[12pt]{amsart}
\usepackage{amssymb, amscd, tikz, latexsym, dsfont,amsmath,amsthm}

\newcommand{\field}[1]{\mathbb{#1}}

\newcommand{\NN}{\field{N}}

\newcommand{\TT}{\field{T}}
\newcommand{\ZZ}{\field{Z}}
\newcommand{\QQ}{\field{Q}}

\newcommand{\Cc}{\mathcal C}

\newcommand{\Kk}{\mathcal K}
\newcommand{\Ll}{\mathcal L}

\newcommand{\Oo}{\mathcal O}
\newcommand{\Pp}{\mathcal P}
\newcommand{\Qq}{\mathcal Q}
\newcommand{\Tt}{\mathcal T}
\newcommand\3[1]{{\mathds #1}}

\newcommand{\NO}{\operatorname{\mathcal{NO}}}

\newcommand{\NT}{\operatorname{\mathcal{NT}}}

\newcommand{\id}{\operatorname{id}}

\newcommand{\lsp}{\operatorname{span}}
\newcommand{\clsp}{\operatorname{\overline{span\!}\,\,}}

\newcommand{\Tc}{\Tt_{\rm cov}}

\theoremstyle{plain}
\newtheorem{theorem}{Theorem}[section]
\newtheorem*{theorem*}{Theorem}
\newtheorem*{prop*}{Proposition}
\newtheorem{cor}[theorem]{Corollary}
\newtheorem{lemma}[theorem]{Lemma}
\newtheorem{prop}[theorem]{Proposition}

\theoremstyle{remark}
\newtheorem{rmk}[theorem]{Remark}

\theoremstyle{definition}

\numberwithin{equation}{section}

\begin{document}

\title{The Cuntz algebra $\Qq_\NN$ and $C^*$-algebras of product systems}

\author{Jeong Hee Hong}\address{Department of Data Information, Korea Maritime University,
Busan 606--791, South Korea}
\email{hongjh@hhu.ac.kr}

\author{Nadia S. Larsen}
\address{Department of Mathematics, University of Oslo, PO Box 1053 Blindern,
N--0316 Oslo, Norway}
\email{nadiasl@math.uio.no}

\author{Wojciech Szyma{\'n}ski}
\address{Department of Mathematics and Computer Science, University of Southern Denmark,
Campusvej 55, DK--5230 Odense M, Denmark}
\email{szymanski@imada.sdu.dk}

\thanks{J. H. Hong was supported by Basic Science Research
Program through the National Research
Foundation of Korea (NRF) funded by the Ministry of Education, Science and Technology (2010--0022884). W. Szyma{\'n}ski 
was partially supported by the FNU
Forskningsprojekt `Structure and Symmetry'. N. Larsen was supported by the Research Council of Norway. The last two authors 
were also supported by the EU-Network "Noncommutative Geometry" (Contract No. MRTN-CT-2006-031962) and the NordForsk research network  
"Operator Algebra and Dynamics".} 

\begin{abstract}
We consider a product system over
the multiplicative semigroup $\NN^\times$ of Hilbert bimodules
which is implicit in work of S. Yamashita and of the second named author.
We prove directly, using universal properties, that the associated Nica-Toeplitz
algebra is an extension of the $C^*$-algebra $\Qq_\NN$ introduced recently by
Cuntz.
\end{abstract}

\date{August 1, 2011}
\vskip 1cm

\maketitle

\section{Introduction}\label{section:Intro}

In 2006, Cuntz (\cite{Cun1}) initiated investigations of purely infinite and simple $C^*$-algebras
$\Qq_R$ associated to the $ax+b$-semigroup $\NN\rtimes \NN^\times$ over the natural numbers $\NN$ and, more generally, $R\rtimes R^\times$ over a ring R. 
The algebra $\Qq_\NN$ is a crossed product of the Bunce-Deddens
algebra associated to $\QQ$ by the action of the multiplicative semigroup $\NN^\times$, and  $\Qq_\NN$
is also generated by the Bost-Connes $C^*$-algebra $\Cc_\QQ$ (\cite{BC}) and one more unitary generator.
An analogous construction based on $R=\ZZ$ instead led to a purely infinite $C^*$-algebra
$\Qq_\ZZ$ such that $\Qq_\ZZ\cong\Qq_\NN\rtimes\ZZ_2$. Shortly afterwards, this work of Cuntz was
further extended and generalized to the context of integral domains by Cuntz
and Li, \cite{CunLi}, and to arbitrary rings by Li, \cite{Li}.

As it turns out, Cuntz's algebra  $\Qq_\NN$ can be usefully viewed in several different ways. Firstly, in \cite[Example 5.2]{Yam} Yamashita constructs a
topological $k$-graph $\Lambda$  (with $k=\infty$) and shows that the
corresponding graph $C^*$-algebra $C^*(\Lambda)$ is isomorphic to
$\Qq_\NN$. Secondly, Laca and Raeburn demonstrated that
$\Qq_\NN$ arises as a quotient of the Toeplitz algebra corresponding to the quasi-lattice ordered group 
$(\QQ\rtimes\QQ_+^*, \NN\rtimes\NN^\times)$, \cite{Lac-Rae2}. Thirdly, since $\Qq_\NN$ is purely infinite, 
simple and has free-abelian $K_1$-group, it
follows from the Kirchberg-Phillips classification and \cite{Szy} that its stabilization is isomorphic
to a graph $C^*$-algebra. We think it would be very interesting to find an explicit form of such an
isomorphism, but this has not been achieved yet.

Yamashita's approach leading to the claimed isomorphism of $\Qq_\NN$ with an algebra of form
$C^*(\Lambda)$ consists of constructing a row-finite topological $k$-graph
$\Lambda$ with $k=\infty$ in the sense of \cite{Ye}. We wanted to fill out
details in Yamashita's construction. Soon enough we saw that there was a
product system around which was related to work of the second named author
generalizing Exel's crossed product to abelian semigroups, \cite{Lar}.

Our initial motivation for this project was to
understand the structure of the $C^*$-algebras associated to this product system $X$
over $\NN^\times$, and to look into the analysis of KMS states of the universal
$C^*$-algebra $\Tc(X)$ for Nica covariant Toeplitz representations as constructed
by Fowler, \cite{Fow2}. We give an explicit and detailed description of $X$, and
prove a number of relevant properties.  Our main  result,
Theorem~\ref{thm:presentation_T(XA)}, gives a presentation of $\Tc(X)$ in terms of
generators and relations. A consequence of this is that the Cuntz-Pimsner algebra
$\Oo_X$ in Fowler's sense is isomorphic to $\Qq_\NN$.

While we were working on this project it turned out that Brownlowe, an Huef, Laca
and Raeburn \cite{BanHLR} were studying the same product system and the  relations of the associated
$C^*$-algebras to both $\Qq_\NN$ and Laca-Raeburn's Toeplitz algebra $\Tt(\NN\rtimes
\NN^\times)$ of the affine semigroup over the natural numbers from \cite{Lac-Rae2}.
Their approach is somewhat
different and relies on characterising faithful representations of what they call the
additive boundary of $\Tt(\NN\rtimes \NN^\times)$.
We believe that our direct approach, based on constructing Nica covariant
representations, can be useful in understanding more general product systems along
similar lines. We plan to take up the analysis of KMS states on more general algebras $\Tc(X)$ in a
future paper.

Other recent works which recast $\Qq_\NN$ and its generalizations to rings $R$ are \cite{KLQ} and \cite{BE}.

We mention that the main result of the present note was first announced at the
conference on ``Selected topics in Operator Algebras and Non-commutative Geometry'' in Victoria, Canada, in July
2010, where the similar result from \cite{BanHLR} was also announced.

\section{Preliminaries}\label{section:Prelim}

\subsection{Product systems of Hilbert bimodules}\label{Prelim-1}

Let $A$ be a $C^*$-algebra and $X$ be a complex vector space with a right action of $A$.
Suppose that there is a $A$-valued inner product
$\langle \cdot , \cdot \rangle _A$ on $X$ which is conjugate linear in the first variable
and  satisfies
\begin{enumerate}
\item $\langle\xi, \eta\rangle _A =\langle\eta , \xi\rangle_A^*$,
\item $\langle\xi, \eta\cdot a\rangle_A=\langle\xi, \eta\rangle_A\, a$,
\item $\langle\xi, \xi\rangle_A \geq 0 $ and $\langle\xi, \xi\rangle_A=0 \; \Longleftrightarrow\; \xi=0$,
\end{enumerate}
for $\xi,\eta\in X$ and $a\in A$. Then $X$ becomes a right Hilbert $A$-module
when it is complete with respect to the norm given by
$\|\xi\|:=\|\langle\xi,\xi\rangle _A\|^{\frac{1}{2}}$ for $\xi\in X$.

\begin{rmk}
In this paper, we will use Exel's method \cite{Ex} (see also \cite{Bro}, \cite{Lar-Rae},
\cite{Kw-Le} and \cite{Lar}) of constructing $C^*$-valued inner products via transfer operators.
Namely, if $\alpha$ is an endomorphism of
a unital $C^*$-algebra $A$, a {\it transfer operator} for $\alpha$ is a positive continuous linear map $L:A\to A$
satisfying $L(a\alpha (b))=L(a)b$ for $a, b\in A$. If a right $A$-module $X$ is equipped
with a right action $\xi\cdot a=\xi\alpha (a)$, then a $A$-valued pre-inner product on $X$ may
be defined by $\langle \xi,\eta\rangle_A:=L(\xi^*\eta)$ for $\xi, \eta \in X$.
\end{rmk}

A map $T:X\rightarrow X $ is said to be adjointable if there is a map $T^*:X\rightarrow X$
such that $\langle T\xi,\zeta\rangle_A=\langle\xi,T^*\zeta\rangle_A$ for all $\xi,\eta\in X$. An
adjointable map (or operator) is linear and norm-bounded, and the set $\Ll(X)$ of all adjointable
operators on $X$ endowed with the operator norm
is a $C^*$-algebra. The rank-one operator $\theta_{\xi, \eta}$ defined on $X$ as
$$
\theta_{\xi,\eta}(\zeta)=\xi\cdot\langle \eta,\zeta\rangle_A \; \; \text{for} \; \xi, \eta, \zeta\in X,
$$
is adjointable and we have $\theta_{\xi,\eta}^*=\theta_{\eta,\xi}$. Then
$\Kk(X)=\clsp\{\theta_{\xi, \eta}\mid \xi, \eta\in X\}$ is the ideal of (generalized)
compact operators in $\Ll(X)$. We let $I_Y$ denote the identity element in $\Ll(X)$.

\vspace{2mm}
Suppose $X$ is a right Hilbert $A$-module. A $*$-homomorphism $\phi:A\rightarrow\Ll(X)$ induces a left action of $A$
on a $X$ by $a\cdot\xi=\phi(a)\xi$, for $a\in A$ and $\xi\in X$.
Then $X$ becomes a right-Hilbert $A$--$A$-bimodule (we mention that the terminology $C^*$-correspondence over $A$ is also used).
The standard bimodule $_AA_A$ is
equipped with $\langle a, b\rangle _A=a^*b$, and the right and left actions are simply given
by right and left multiplication in $A$, respectively.

For right-Hilbert $A$--$A$-bimodules $X$ and $Y$, the (balanced) tensor product $X\otimes_AY$ becomes
a right-Hilbert $A$--$A$-bimodule with the right action from $Y$,  the left action implemented by the
homomorphism $A\ni a \mapsto\phi(a)\otimes_A I_Y\in\Ll(X\otimes_A Y)$, and the $A$-valued
inner product given by $\langle\xi_1\otimes_A\eta_1 , \xi_2\otimes_A\eta_2\rangle_A=\langle \eta_1, \langle\xi_1,
\xi_2\rangle_A\cdot\eta_2\rangle_A$, 
for $\xi_i\in X$ and $\eta_i\in Y$, $i=1,2$.

\vspace{2mm}
Let $P$ be a multiplicative semigroup with identity $e$, and let $A$ be a $C^*$-algebra.
For each $p\in P$ let $X_p$ be a complex vector space.
Then the disjoint union $X := \bigsqcup_{p\in P}X_{p}$ is a {\em product system} over $P$
if the following conditions hold:
\begin{enumerate}\renewcommand{\theenumi}{P\arabic{enumi}}
\item\label{it:P-1} For each $p\in P\setminus\{e\}$, $X_p$ is a right-Hilbert $A$--$A$-bimodule.
\item\label{it:P-2}$X_e$ equals the standard bimodule $_AA_A$.
\item\label{it:P-3} $X$ is a semigroup such that $\xi\eta\in X_{pq}$ for $\xi\in X_p$ and $\eta\in X_q$,
and for $p, q\in P\setminus\{e\}$, this product extends to an isomorphism
$F^{p, q} : X_p\otimes_A X_q\rightarrow X_{pq}$ of right-Hilbert $A$--$A$-bimodules. If $p$ or $q$ equals $e$
then the corresponding product in $X$ is induced by the left or the right action of $A$.
\end{enumerate}

\begin{rmk}
For $p\in P$, the multiplication on $X$ induces maps
$F^{p, e}:X_p\otimes_A X_e\rightarrow X_p$ and $F^{e, p}:X_e\otimes_A X_p\rightarrow X_p$
by multiplication $F^{p, e}(\xi\otimes a)=\xi\, a$ and $F^{e, p}(a\otimes\xi)=a\, \xi$
for $a\in A$ and $\xi\in X_p$. Note that $F^{p,e}$ is an isomorphism. However, $F^{e,p}$ is
an isomorphism if $\overline{\phi(A)X_p}=X_p$ or, in the terminology from \cite{Fow2}, if $X_p$ is
essential.
\end{rmk}

For each $p\in P$, we denote by $\langle\cdot,\cdot\rangle_p$ the $A$-valued
inner product on $X_p$  and by $\phi_p$ the homomorphism from $A$ into $\Ll(X_p)$.
Due to associativity of the multiplication on $X$, we have
$\phi_{pq}(a)(\xi\eta)=(\phi_p(a)\xi)\eta$ for all $\xi\in X_p$, $\eta\in X_q$, and $a\in A$.

 For each pair $p,q\in P\setminus{e}$,
the isomorphism $F^{p,q} : X_p\otimes_A X_q\rightarrow X_{pq}$
allows us to define a $*$-homomorphism $i_p^{pq}:\Ll(X_{p})\to\Ll(X_{pq})$ by
$i_p^{pq}(S)=F^{p,q}(S\otimes_A I_q)(F^{p,q})^*$ for $S\in\Ll(X_p)$.
In the case $r\neq pq$ we define $i_p^r:\Ll(X_p)\to\Ll(X_r)$ to
be the zero map $i_p^r(S)=0$ for all $S\in\Ll(X_p)$. Further, we let $i_e^q=\phi_q$.

\vspace{2mm}
Many interesting product systems arise over semigroups equipped with additional structures.
In \cite{Ni}, $(G, P)$ is called a
quasi-lattice ordered group if (i) $G$ is a discrete group, (ii) $P$ is a sub-semigroup of $G$ with $P\bigcap P^{-1}=\{e\}$, (iii) with respect to the order
$p\preceq q\; \Leftrightarrow\; p^{-1}q\in P$, every two elements $p, q\in G$ which have
a common upper bound in $P$ have a least upper bound $p\vee q\in P$. If this is the
case we write $p\vee q<\infty$. Here we are interested in lattice-ordered
pairs for which $p\vee q<\infty$ for all $p,q\in P$.

Assuming $X$ is a product system over $P$ with $(G,P)$ a quasi-lattice ordered group, there
naturally arises a certain property related to compactness.
A product system $X=\sqcup_{p\in P}X_p$ is called {\it compactly aligned} if
$i_p^{p\vee q}(S)i_q^{p\vee q}(T)\in\Kk(X_{p\vee q})$ for all $p, q\in P$ with $p\vee q<\infty$
and $S\in\Ll(X_p)$ and $T\in\Ll(X_q)$, \cite{Fow2}. Note that in general neither $i_p^{p\vee q}(S)\in\Kk(X_p)$
nor $i_q^{p\vee q}(T)\in\Kk(X_q)$
is required.


\subsection{$C^*$-algebras associated to product systems}\label{Prelim-2}

Let $(G,P)$ be a quasi-lattice ordered group, $A$ a $C^*$-algebra, and
$X=\sqcup_{p\in P}X_p$  a product system over $P$ of right-Hilbert $A$--$A$-bimodules. A map
$\psi$ from $X$ to a  $C^*$-algebra $C$ is a Toeplitz
representation of $X$ if the following conditions hold:
\begin{enumerate}\renewcommand{\theenumi}{T\arabic{enumi}}
\item\label{it:T-1} for each $p\in P\setminus\{e\}$, $\psi_p:=\psi\vert_{X_p}$ is linear,
\item\label{it:T-2} $\psi_e:A\longrightarrow C$ is a $*$-homomorphism,
\item\label{it:T-3} $\psi_p(\xi)\psi_q(\eta)=\psi_{pq}(\xi\eta)\; \; $ for $ \; \xi\in X_p$, $\eta\in X_q$, $p,q\in P$,
\item\label{it:T-4} $\psi_p(\xi)^*\psi_p(\eta)=\psi_e(\langle\xi, \eta\rangle_p)$ for $\xi, \eta\in X_p$.
\end{enumerate}
As shown in \cite{P}, for each $p\in P$ there exists a $*$-homomorphism
$\psi^{(p)} : \Kk(X_p) \longrightarrow C$ such that $\psi^{(p)}(\theta_{\xi,\eta})=
\psi_p(\xi)\psi_p(\eta)^*\, , \; \text{for}\; \xi,\eta\in X_p$. The Toeplitz representation $\psi$ is
\begin{enumerate}
\item {\it Cuntz-Pimsner covariant} \cite{Fow2} if $\psi^{(p)}(\phi_p(a))=\psi_e(a)\;
\text{for}\; a\in\phi_p^{-1}(\Kk(X_p))$ and all $p\in P$;
\item {\it Nica covariant} \cite{Fow2} if
$$\psi^{(p)}(S)\psi^{(q)}(T)=
    \begin{cases}\psi^{(p\vee q)}(i^{p\vee q}_p(S)i^{p\vee q}_q(T)), &\; \text{if} \; p\vee q <\infty \\
        0 \; ,& \; \; \text{otherwise}
    \end{cases}
$$
for $S\in {\mathcal K}(X_p)$, $T\in {\mathcal K}(X_q)$, $p,q\in P$, provided that $X$ is compactly aligned.
\end{enumerate}

The Toeplitz algebra $\Tt(X)$ associated to the product system $X$ was defined by Fowler as  the universal
$C^*$-algebra for Toeplitz representations, see \cite{Fow2}.
Similarly, the Cuntz-Pimsner algebra $\Oo(X)$ is universal
for the Cuntz-Pimsner covariant Toeplitz representations. In \cite[\S6]{Fow2}, Fowler
introduced a $C^*$-algebra $\Tc(X)$ as a subalgebra of a certain crossed product by $X$, and in
\cite[Theorem 6.3]{Fow2},
he showed that  $\Tc(X)$ is universal for Nica covariant Toeplitz
representations of $X$ on Hilbert space (the definition of such representations is
\cite[Definition 5.1]{Fow2}). It follows from \cite[Theorem 6.3]{Fow2} that for a compactly aligned
product system $X$ over $P$ of essential right-Hilbert $A$--$A$-bimodules, $\Tc(X)$ is universal for the $C^*$-algebraic
version of Nica covariance. It was pointed out in \cite{CLSV} that one can
drop the assumption on each $X_p$ being essential. We let $\iota$ denote the universal Nica covariant Toeplitz
representation of the compactly aligned product system $X$.

Given $X$ compactly aligned, the Cuntz-Nica-Pimsner algebra $\NO_X$ is universal for the Cuntz-Nica-Pimsner
covariant Toeplitz representations introduced in \cite{SY}.
Sims and Yeend's definition of a CNP covariant representation is very technical,
and we do not recall it here. We merely mention
that in general, $\Tc(X)$ is a quotient of $\Tt(X)$ and $\NO_X$ is a quotient of $\Tc(X)$.
 In some situations, $\NO_X$ coincides with $\Oo(X)$, see \cite{SY} for details
and further discussion.

\begin{rmk}
Regarding notation, it was argued  in \cite[Remark 5.3]{BanHLR} that the choice
of  $\Tt(X)$ and  $\Tc(X)$ for $C^*$-algebras generated by
universal representations (with some properties) was unfortunate, because a Toeplitz algebra
of some sort should be generated by the Fock representation of the system.  So instead of
$\Tc(X)$ one may also use the notation $\NT(X)$ of \cite{BanHLR}.
\end{rmk}


\subsection{Cuntz's $\Qq_\NN$ and Laca-Raeburn's $\Tt(\NN\rtimes\NN^\times)$}

In \cite{Cun1}, Cuntz introduced $\Qq_\NN$, the universal $C^*$-algebra generated by a
unitary $u$ and isometries $s_n$, $n\in \NN^\times$, subject to the relations
\begin{enumerate}\renewcommand{\theenumi}{Q\arabic{enumi}}
\item\label{it:Qn-1} $s_ms_n=s_{mn}$,
\item\label{it:Qn-2} $s_m u=u^m s_m$, and
\item\label{it:Qn-3} $\sum_{k=0}^{m-1}u^ks_m s_m^*u^{-k}=1,$
\end{enumerate}
for all $m,n\in \NN^\times$. Cuntz proved that $\Qq_\NN$ is simple and purely
infinite.

In \cite{Lac-Rae2}, Laca and Raeburn studied the semidirect product $\QQ\rtimes \QQ_+^*$ arising from the action of
 $\QQ_+^*$ by multiplication on the additive group $\QQ$. They showed that the pair $(\QQ\rtimes \QQ_+^*, \NN\rtimes \NN^\times)$
 is a quasi-lattice ordered group. For a quasi-lattice ordered group $(G, P)$, the Toeplitz algebra $\Tt(P)$ is generated by
 the operators $T_p$ on $l^2(P)$ given by $T_p\varepsilon_q=\varepsilon_{pq}$ on the
 canonical orthonormal basis $\{\varepsilon_p\}$. By \cite{Lac-Rae2}, $\Tt(\NN\rtimes\NN^\times)$
 is generated by isometries $s$ and $v_p$ for $p\in \Pp$
(where $\Pp$ denotes the collection of all positive prime integers) subject to
the relations
\begin{enumerate}\renewcommand{\theenumi}{LR\arabic{enumi}}
\item\label{it:LR-1} $v_ps=s^pv_p$,
\item\label{it:LR-2} $v_pv_q=v_qv_p$,
\item\label{it:LR-3} $v_p^*v_q=v_qv_p^*$ when $p\ne q$,
\item\label{it:LR-4} $s^*v_p=s^{p-1}v_p s^*$, and
\item\label{it:LR-5} $v_p^* s^k v_p=0$ for $1\leq k<p$.
\end{enumerate}
This readily implies that $\Qq_\NN$ is a quotient of $\Tt(\NN\rtimes\NN^\times)$ by the ideal generated by $ss^*-1$.


\section{A product system over $\NN^\times$ with fibers $C(\TT)$}\label{section:X}

Let $A$ be the $C^*$-algebra $C(\TT)$. We aim to define a product system
$X$ over $\NN^\times$ of right Hilbert $A$--$A$-bimodules whose
Cuntz-Pimsner algebra is isomorphic to $\Qq_\NN$ and whose Toeplitz algebra is a quotient
of Laca and Raeburn's $\Tt(\NN\rtimes\NN^\times)$.

In all that follows, we let $Z(z)=z$  be the standard unitary generator of $A$.
For each $m\in \NN^\times$ and $f\in C(\TT)$, let $f_m$ denote the function $z\mapsto f(z^m)$
in $C(\TT)$. The map $\alpha_m:f\mapsto f_m$ is then an endomorphism of $C(\TT)$, and
 $L_m:A\to A$ defined by $L_m(f)(z)=\frac 1m\sum_{w^m=z}f(w)$ is a transfer operator for $\alpha_m$
in the sense of \cite{Ex}, that is $L_m$ is positive, linear and continuous, and satisfies
$L_m(f\alpha_m(g))=L_m(f)g$ for all $f,g\in A$.
Let $X_m^0$ be the $A$-module based on $C(\TT)$ as  vector space with right action
$\xi\cdot f=\xi\alpha_m(f)$ for $\xi, f\in A$.
It follows from \cite[Lemma 3.3]{Lar-Rae} that $X_m^0$
is complete in the norm induced by the $A$-valued pre-inner product $\langle \xi,\eta\rangle_m=L_m(\xi^*\eta)$.
We can define a left action of $A$ on $X_m^0$ by pointwise multiplication $f\cdot \xi=f\xi$ for $f,\xi \in A$.
Hence $X_m^0$ becomes a right-Hilbert $A$--$A$-bimodule which as a vector space is just $A$. To distinguish
the copies of $A$ corresponding to different $m,n$ in $\NN^\times$ we relabel
$X_m^0$ as $X_{m}$ and write its elements as $\xi\31_m$ with $\xi\in A$.

Thus $X_{m}$ is a right Hilbert $A$--$A$-bimodule with the right action
\begin{equation}\label{eq:right-action-m}
(\xi\31_m)\cdot f =\xi \alpha_m(f)\31_m\text{ for }\xi\31_m\in X_m\text{ and }f\in A,
\end{equation}
inner product given by
\begin{equation}\label{eq:inner-product-m}
\langle \xi\31_m,\eta\31_m\rangle_m=L_m(\xi^*\eta)
\end{equation}
for $\xi, \eta\in A$, and left action
\begin{equation}\label{eq:left-action-m}
f \cdot \xi\31_m= \phi_m(f)(\xi\31_m)= (f\xi)\31_m
\end{equation}
for $f, \xi\in A$. Note that the left action $\phi_m$ is injective for each $m\in\NN^\times$,
and that $\phi_m(1)=I_m$. Whence, in
particular, $\phi_m(A)X_m=X_m$ (that is, the Hilbert
bimodule $X_m$ is essential). Furthermore,
$X_{A,1}$ is identical with the standard bimodule $_A A_A$. We obtain a product system
\begin{equation}\label{eq:xb}
X := \bigsqcup_{m\in\NN^\times}X_m
\end{equation}
with multiplication $X_m\times X_r\to X_{mr}$ given by
\begin{equation}\label{eq:product}
(\xi\31_m)(\eta\31_r):=(\xi\alpha_m(\eta))\31_{mr}
\end{equation}
for $m,r\in\NN^\times$, see also \cite{BanHLR}. We claim that the map
(\ref{eq:product}) extends to an isomorphism of right-Hilbert $A$--$A$-bimodules
\begin{equation}\label{eq:def_prod_A}
F^{m,r}:X_m\otimes_A X_r\to X_{mr}
\end{equation}
for all $m,r\in \NN^\times$. To this end, we first notice that the map defined in (\ref{eq:product})
is bilinear and $A$-balanced and thus extends to a linear map $F^{m,r}$ as in (\ref{eq:def_prod_A}).
A straightforward calculation shows that $F^{m,r}$ is adjointable, with adjoint
$(F^{m,r})^*:X_{mr}\to X_m\otimes_A X_r$ given by
\begin{equation}\label{eq:Fadjoint}
(F^{m,r})^*(\zeta\31_{mr})=(\zeta\31_m)\otimes_A \31_r.
\end{equation}
To argue that $F^{m,r}$ is a unitary isomorphism of right $A$-modules we need to show that $\langle F^{m,r}(\xi_1\otimes_A \eta_1), F^{m,r}(\xi_2\otimes_A
\eta_2)\rangle_{mr}=\langle\eta_1, \langle \xi_1,\xi_2\rangle_m\cdot \eta_2\rangle_r$
for all $\xi_i\in X_{m}$ and $\eta_i\in X_{A, r}$,
$i=1,2$. This amounts to proving that $L_{mr}=L_r\circ L_m$, a fact which may be easily
verified.
In the terminology of \cite{Lar}, the transfer operators
$L_m$ form an action of $\NN^\times$ on $A$.

 It is also clear from the definition
that $F^{m,r}$ commutes with the left action of $A$.
As noted, for each pair $m,r\in\NN^\times$ we define an embedding
$i_m^{mr}:\Ll(X_m)\to\Ll(X_{mr})$ as
\begin{equation}\label{eq:imr}
i_m^{mr}(S)=F^{m,r}(S\otimes \id_r)(F^{m,r})^*.
\end{equation}

Before discussing the $C^*$-algebras associated to $X$ we prove
the following simple but useful lemma. We first introduce some terminology. Fix $m\in \NN^\times$, and consider the
action $\rho_m$ by rotations of the cyclic group $\ZZ_m$ on $A$, thus $\rho_m(j)(Z)=e^{2\pi i j/m}Z$. Let
$A\rtimes\ZZ_m$ be the corresponding crossed product. We identify $g\in \ZZ_m$
with the corresponding unitary in $A\rtimes \ZZ_m$ implementing the rotation by $e^{2\pi i g/m}$. We view $e_m=\sum_{g\in\ZZ_m}g$ as
an element of $A\rtimes \ZZ_m$. It is known that there is a conditional
expectation
\begin{equation}\label{eq:cond_exp_A}
E_m: A\to A^\rho, \,\,E_m(a)=\frac 1m\sum_{j=0}^{m-1}\rho_m(j)(a)
\end{equation}
for $a\in A$. Then $E_m(Z^kZ^{*l})$ is zero when  $k\not\equiv l\;(\operatorname{mod} m)$
and is $Z^kZ^{*l}$ otherwise.

\begin{lemma}\label{lem:leftact_compacts}
For each $m\in\NN^\times$ we have
\begin{enumerate}
\item $\phi_m(A)\subseteq \Kk(X_m)$, and
\item the linear span of $\langle X_m,\,X_m \rangle _m$ is dense in $A$ (that is, $X_m$ a full $A$-module).
\end{enumerate}
\end{lemma}
\begin{proof}
Ad (1). The ideal of $A\rtimes\ZZ_m$ generated by $e_m$
coincides with $\clsp Ae_m A$. Since the action is free it is also saturated, \cite{Ph}.
Thus $\lsp Ae_m A=A\rtimes \ZZ_m$ and, consequently, $A \subseteq\lsp A e_m A$.

Now we observe that the pair $(\phi_m,\,U)$, with $(U_g \xi)(z)=\xi(g^{-1}z)$,  is
a covariant representation for this dynamical system in $\Ll(X_m)$. For $f,h\in C(\TT)$, the
corresponding representation of the crossed product sends $fe_m \overline{h}$ to
$\theta_{f,h}$. Thus the image of the ideal generated by $e_m$ coincides with  $\Kk(X_m)$.
It follows that $\phi_m(A)\subseteq \Kk(X_m)$, as required.

Ad (2). This follows from part (1) of the lemma, for otherwise $\clsp(\langle X_m,\,X_m \rangle _m)$
would be a proper ideal of $A$ and thus $\Kk(X_m)$ would not contain the identity operator.
\end{proof}

Now we move to describing $C^*$-algebras associated to $X$ from \eqref{eq:xb}. Here the
quasi-lattice ordered group $(\QQ_+^*, \NN^\times)$ is lattice ordered:
every pair of elements $(m, n)$ admits
a least upper bound equal to their least common multiple $m\vee n$. Further,
the bimodules $X_m$ have good properties: they are essential and full, and
the left action of $A$ is by compact operators, cf. Lemma
\ref{lem:leftact_compacts}. It follows from \cite[Proposition 5.8]{Fow2} that $X$ is compactly aligned, and
\cite[Theorem 6.3]{Fow2} shows that $\Tc(X)=\clsp\{\imath(\xi)\imath(\eta)^*\mid \xi\in X_m, \,
\eta\in X_n, m,n\in \NN^\times\}$.

\begin{rmk}\label{densesubalgebra}
As shown by Fowler, if  $\xi,\eta\in X$ then
$\imath(\xi)^*\imath(\eta)$ can be approximated by linear combinations of
elements of $\imath(X)\imath(X)^*$. However, a closer inspection of the proof
of \cite[Proposition 5.10]{Fow2} reveals that if the compact operators in each fiber contain the
corresponding identity operator then $\imath(\xi)^*\imath(\eta)$ is itself a linear combination of
elements of $\imath(X)\imath(X)^*$. This implies that in the presently considered
case $\lsp\{\imath(\xi)\imath(\eta)^*\mid \xi\in X_m,\, \eta\in X_n, m,n\in \NN^\times\}$
is not merely a dense self-adjoint subspace of $\Tc(X)$ but a dense $*$-subalgebra.
\end{rmk}

We set $i_m:=\imath\vert_{X_m}$ for $m\in \NN^\times$
(in particular, $i_1=\imath\vert_A$), and denote by $i^{(m)}$ the
corresponding homomorphism of $\Kk(X_m)$. The main result of this paper is the following
characterization of $\Tc(X)$.

\begin{theorem}\label{thm:presentation_T(XA)}
The element
$u :=i_1(Z)$ is a unitary in $\Tc(X)$, each of $w_m=i_m(\31_m)$
for $m\in \NN^\times$ is an isometry, and $u,w_m$  satisfy the relations
\begin{enumerate}\renewcommand{\theenumi}{B\arabic{enumi}}
\item\label{it:ToeB0} $w_{mn}=w_mw_n$ for all $m, n\in \NN^\times$,
\item\label{it:ToeB1} $w_mu=u^mw_m$ for all $m\in \NN^\times$,
\item\label{it:ToeB2} $w_p^*w_q=w_qw_p^*$ if $p,q$ distinct primes, and
\item\label{it:ToeB3} $w_p^*u^k w_p=0$ if $p$ is prime and $1\leq k<p$.
\end{enumerate}

Furthermore, $\Tc(X)$ is the universal $C^*$-algebra for the relations
\eqref{it:ToeB0}--\eqref{it:ToeB3}.
\end{theorem}

We mention that the same presentation of $\Tc(X)$ was identified in \cite[Theorem 5.2]{BanHLR}. The method of proof there is different,
and uses the characterization of faithful representations on the additive boundary of $\Tt(\NN\rtimes \NN^\times)$, which
is shown to be $\Tc(X)$.

To prove the theorem we will need to understand what Nica-covariance means for the product
system $X$. For this we first derive a number of consequences of the
relations
\eqref{it:ToeB0}--\eqref{it:ToeB3}.

\begin{lemma}\label{lem:consequences_of_relations}
Assume the relations \eqref{it:ToeB0} and \eqref{it:ToeB1}. Then
\begin{enumerate}
\item\label{it:ToeB1-star} $w_mu^*=u^{*m}w_m$
\item\label{it:ToeB2-star} $u^lw_m^*=w_m^*u^{lm}$ for all $l\in \NN$.
\item\label{it:ToeB3-mn} If also \eqref{it:ToeB2} is satisfied, then
$w_mw_n^*=w_n^*w_m$ if $gcd(m, n)=1$.
\end{enumerate}
\end{lemma}

\begin{proof}
Taking adjoints on both sides of \eqref{it:ToeB1} and using that $u$ is a
unitary gives $uw_m^*=w_m^*u^m$. From this relation
\eqref{it:ToeB1-star} follows upon taking adjoints, and
\eqref{it:ToeB2-star} by induction. For \eqref{it:ToeB3-mn} we write $m,n$
in prime factorization  and apply repeatedly relation \eqref{it:ToeB2}.
\end{proof}

\begin{lemma}\label{lem:more_consequences_of_relations}
Assume the relations \eqref{it:ToeB0}, \eqref{it:ToeB1} and \eqref{it:ToeB3}.
For $m,n\in \NN^\times$ and $l\in \ZZ$
\begin{equation}\label{eq:simplification}
w_m^*u^lw_n=\begin{cases}0&\text{ if } l\not\equiv 0(\text{mod }gcd(m,n))\\
w_{n^{-1}(m\vee n)}^*u^{l/gcd(m,n)}w_{m^{-1}(m\vee n)}&\text{ otherwise}.\end{cases}
\end{equation}
\end{lemma}

\begin{proof}
If $gcd(m,n)$ does not divide $l$, there is a prime divisor $p$ of $gcd(m,n)$
such that for unique integers $l_1$ and $r$ with $1<r<p$ we have $l=l_1 p+r$.
Writing $m=pm_1$ and $n=pn_1$, assuming $l\geq 0$ and using \eqref{it:ToeB3} in the last
equality gives
$$
w_m^*u^lw_n=w_{pm_1}^*u^ru^{l_1p}w_{pn_1}=w_{m_1}^*w_p^*u^rw_pu^{l_1}w_{n_1}=0.
$$
By Lemma~\ref{lem:consequences_of_relations} \eqref{it:ToeB2-star}, the case $l<0$ follows.
If $gcd(m,n)$ divides $l$,
Lemma~\ref{lem:consequences_of_relations} \eqref{it:ToeB2-star} and the fact that $m/gcd(m,n)=
n^{-1}(m\vee n)$, $n/gcd(m,n)=m^{-1}(m\vee n)$ imply the claim.
\end{proof}

\begin{cor}
Assume the relations \eqref{it:ToeB0}--\eqref{it:ToeB3}.
Given $m,n\in \NN^\times$ write $d=gcd(m,n)$, let $m'=m/d$ and $n'=n/d$,
and choose integers $\alpha',\beta'$  such that $1=\alpha'm'-\beta'n'$.
Then for $k, l\in \ZZ$ we have
\begin{equation}\label{eq:more_relations}
w_mw_m^*u^{*k}u^lw_nw_n^*
=\begin{cases}
0&\text{ if }l\not\equiv k(\text{mod }gcd(m,n))\\
u^{m\alpha'(l-k)/d}w_{m\vee n}w_{m\vee n}^*u^{* (n\beta' (l-k)/d)}
&\text{ otherwise}.
\end{cases}
\end{equation}
\end{cor}

\begin{proof}
The equality giving $0$ when $d$ does not divide $l-k$
is immediate from \eqref{eq:simplification}.
So we assume that $d$ divides $l-k$ and compute the left hand-side
of  \eqref{eq:more_relations} as follows
\begin{align*}
w_mw_m^*u^{*k}u^lw_nw_n^*
&=w_mw_{m'}^*u^{(l-k)/d}w_{n'}w_n^* \text{ by }
\eqref{eq:simplification}\\
&=w_mw_{m'}^*u^{\alpha' m'(l-k)/d}u^{*\beta' n'(l-k)/d}w_{n'}w_n^*\\
&=w_mu^{\alpha'(l-k)/d} w_{m'}^*w_{n'}u^{* \beta'(l-k)/d}w_n^* \text{ by }Lemma~\ref{lem:consequences_of_relations}\\
&=w_mu^{\alpha'(l-k)/d} w_{n'}w_{m'}^*u^{*\beta' (l-k)/d}w_n^*
\text{ by }Lemma~\ref{lem:consequences_of_relations}\\
&=u^{m\alpha'(l-k)/d} w_mw_{m^{-1}(m\vee n)}w_{n^{-1}(m\vee n)}^*w_n^*u^{*(n\beta' (l-k)/d)}\\
&=u^{m\alpha'(l-k)/d} w_{m\vee n}w_{m\vee n}^*u^{*(n\beta'(l-k)/d)},
\end{align*}
as needed.
\end{proof}

By \cite[Proposition 2]{Ex}, $\alpha_m\circ L_m$
is a non-degenerate conditional expectation onto $\alpha_m(A)$ for each $m\in \NN^\times$. We prove next
that this expectation is precisely the one constructed in \eqref{eq:cond_exp_A}.

\begin{lemma}\label{lemma:Em_from_alpha_Lm}
\textnormal{(a)} For $m\in \NN^\times$ and $k\in \ZZ$ we have
\begin{equation}\label{eq:alpha_with_L}
(\alpha_m \circ L_m)(Z^k)=E_m(Z^{k})=
\begin{cases}0&\text{ if }k\not\equiv 0(\operatorname{mod}m)\\
Z^k&\text{ if }k\equiv 0(\operatorname{mod}m)
\end{cases}
\end{equation}

\textnormal{(b)} For $m,n\in \NN^\times$ we have $E_m\circ E_n=E_{m\vee n}.$
\end{lemma}

\begin{proof} For part (a) note that
$$
(\alpha_m\circ L_m)(Z^k)(z)=L_m(Z^k)(z^m)=\frac 1m\sum_{w^m=z^m}w^k=(\frac 1m\sum_{w^m=1}w^k)z^k;
$$
now, this is $0$ if $1\leq k<m$ and is $z^k$ if $k\equiv 0\,(\operatorname{mod} m)$, and the
claim follows since $C(\TT)$ is generated by $Z$.

To prove (b), assume first that $gcd(m, n)=1$. By (a), $(\alpha_n\circ L_n)(Z^k)$ is zero unless
$k\equiv 0\,(\operatorname{mod} n)$, and further $(\alpha_m\circ L_m)\circ (\alpha_n\circ L_n)(Z^k)$ is
zero unless also $k\equiv 0\,(\operatorname{mod} m)$.
Since $m\vee n=mn$, the terms on both sides of the equality in (b) are zero simultaneously.
If on the other hand $k$ is divisible
by both $m$ and $n$ the two terms  equal $Z^k$. In the general case let $d=gcd(m, n)$ and put $m'=m/d=n^{-1}(m\vee n)$
and $n'=n/d=m'(m\vee n)$. Then by what we have just done
$(\alpha_{m'}\circ L_{m'})\circ(\alpha_{n'}\circ L_{n'})=
\alpha_{m' n'}\circ L_{m' n'}$. Hence, using the transfer property of
$L_d$ and the fact that $L$ is multiplicative on
$\NN^\times$ we compute that
\begin{align*}
(\alpha_m\circ L_m)\circ(\alpha_n\circ L_n)(f)
&=(\alpha_m\circ L_{m'})\circ L_d(\alpha_d(\alpha_{n'}\circ L_n(f)))\\
&=(\alpha_m\circ L_{m'})\circ(\alpha_{n'}\circ L_n(f))\\
&=\alpha_d(\alpha_{m'n'}\circ L_{m'n'}(L_d(f)))\\
&=(\alpha_{m\vee n}\circ L_{m\vee n})(f)
\end{align*}
for all $f\in A$, as claimed.
\end{proof}

\begin{lemma}\label{lem:Em_compact} Given $m, n\in \NN^\times$ let  $d=gcd(m,n)$. Suppose $f,g\in A$
are such that $E_d(f^*g)=\alpha_{m}(f_0)\alpha_n(g_0)$ for some $f_0, g_0\in A$. Then
for every $h\in A$ we have
$$
E_m(f^*gE_n(h))=\alpha_m(f_0)E_{m\vee n}(\alpha_n(g_0)h).
$$
\end{lemma}

\begin{proof} Indeed, note first that $E_m=E_m\circ E_d$ and $E_d\circ E_n=E_d$ by Lemma~\ref{lemma:Em_from_alpha_Lm} (b).
Thus we
compute
\begin{align}
E_m(f^*gE_n(h))
&=E_m(E_d(f^*g(E_d\circ E_n(h)))=E_m(E_d(f^*g)E_n(h))\label{E-dmn}    \\
&=E_m(\alpha_m(f_0)\alpha_n(g_0)E_n(h))=E_m(\alpha_m(f_0)E_n(\alpha_n(g_0)h))\notag\\
&=\alpha_m(f_0)(E_m\circ E_n(\alpha_n(g_0)h)),\notag
\end{align}
which by Lemma~\ref{lemma:Em_from_alpha_Lm} equals $\alpha_m(f_0)E_{m\vee n}(\alpha_n(g_0)h)$, as claimed.
\end{proof}

Using Lemma~\ref{lem:Em_compact} we will write down explicitly elements
 in $\Kk(X_{m\vee n})$ obtained as $i_m^{m\vee n}(S)i_n^{m\vee n}(T)$ for $S\in \Kk(X_m)$
 and $T\in \Kk(X_n)$.

\begin{cor}\label{lem_prod_of_compacts}
For each $m, n\in \NN^\times$ let $d=gcd(m,n)$ and choose integers
$\alpha'$ and $\beta'$ such that $1=\alpha' m/d -\beta' n/d$.
  Then
\begin{equation}\label{eq:prod_compacts}
i_m^{m\vee n}(\theta_{\31_m, Z^k \31_m})i_n^{m\vee n}(\theta_{Z^l\31_n, \31_n})
=\begin{cases}0&\text{ if }k\not\equiv l(\operatorname{mod}d)\\
\theta_{Z^{m\alpha'(l-k)/d}\31_{m\vee n},Z^{n\beta'(l-k)/d}\31_{m\vee n}}&\text{ if }k\equiv l(\operatorname{mod}d)\end{cases}
\end{equation}
\end{cor}

\begin{proof}  Let $\zeta,\xi,\eta \in A$.  We apply \eqref{eq:imr} to see that
\begin{align}
i_n^{m\vee n}(\theta_{\xi\31_n, \eta \31_n})(\zeta \31_{m\vee n})
&=F^{n, n^{-1}(m\vee n)}(\theta_{\xi\31_n, \eta \31_n}\otimes\id)(F^{n, n^{-1}(m\vee n)})^*(\zeta \31_{m\vee n})\notag \\
&=F^{n, n^{-1}(m\vee n)}(\theta_{\xi\31_n, \eta \31_n}(\zeta\31_n)\otimes_A \31_{n^{-1}(m\vee n)})\notag\\
&=F^{n, n^{-1}(m\vee n)}(\xi(\alpha_n\circ L_n(\eta^*\zeta)\31_n)) \otimes_A \31_{n^{-1}(m\vee n)})\notag\\
&=(\xi(\alpha_n\circ L_n(\eta^*\zeta)))\31_{m\vee n}.\label{eq:n_compact}
\end{align}
We now use \eqref{eq:n_compact} to transform the left-hand side of \eqref{eq:prod_compacts} applied to $\zeta \31_{m\vee n}$ into
$$
\bigl(\alpha_m\circ L_m(Z^{*k}Z^l(\alpha_n\circ L_n)(\zeta))\bigr)\31_{m\vee n}=(E_m(Z^{*k}Z^l(E_n(\zeta))))\31_{m\vee n}.
$$
As in the proof of \eqref{E-dmn}, $E_m(Z^{*k}Z^l(E_n(\zeta)))=E_m(E_d(Z^{l-k})E_n(\zeta))$. Lemma~\ref{lemma:Em_from_alpha_Lm} (a)
says that the last expression is equal to $0$ unless $l-k$ is divisible by $d$. Assuming that $l-k$ is divisible by $d$, we can write
$$E_d(Z^{l-k})=E_d(\alpha_d(Z^{(l-k)/d}))=\alpha_m(Z^{\alpha' (l-k)/d})\alpha_n(Z^{*\beta'(l-k)/d}).
$$
Then Lemma~\ref{lem:Em_compact} implies that $E_m(Z^{*k}Z^l(E_n(\zeta)))=Z^{m\alpha'(l-k)/d}(E_{m\vee n}(Z^{*n\beta'(l-k)/d}\zeta))$.
In all we have $i_m^{m\vee n}(\theta_{\31_m, Z^k\31_n})i_n^{m\vee n}
(\theta_{Z^l\31_n, \31_n})(\zeta \31_{m\vee n})=0$ when $k\not\equiv l\,(\operatorname{mod} d)$, and otherwise
$$
i_m^{m\vee n}(\theta_{\31_m, Z^k\31_n})i_n^{m\vee n}
(\theta_{Z^l\31_n, \31_n})(\zeta \31_{m\vee n})=\theta_{Z^{m\alpha'(l-k)/d}\31_{m\vee n},
Z^{n\beta'(l-k)/d}\31_{m\vee n}}$$
as claimed.
\end{proof}

\begin{rmk}\label{rmk:arbitrary_products_compacts} Note that the
same computations show that when $i, j, k, l$ are
integers then
$$
i_m^{m\vee n}(\theta_{Z^i\31_m, Z^k\31_m})i_n^{m\vee n}(\theta_{Z^l\31_n, Z^j\31_n})=\theta_{Z^{i+m\alpha'(l-k)/d}\31_{m\vee n}, Z^{(j+\beta'(l-k)/d)}\31_{m\vee n}}
$$
precisely when $k\equiv l(\operatorname{mod} (d))$, and is the zero element
in $\Kk(X_{A, m\vee n})$ otherwise.
\end{rmk}

\medskip\noindent
Now we are in a position to complete the proof of our main result.

\medskip
\begin{proof}[Proof of Theorem~\ref{thm:presentation_T(XA)}] Since
$i_1$ is a $*$-homomorphism, $u$ is unitary. In view of
the fact that $\langle \31_m, \31_m \rangle_m = L_m(1)=1$, the Toeplitz relation
$i_1(\langle \31_m, \31_m\rangle_m)=i_m(\31_m)^*i_m(\31_m)$ shows
that each $w_m$ is an isometry.

Ad \eqref{it:ToeB0}. This follows immediately from the identity $\31_{mn}=\31_m\31_n$, which in turn is a consequence
of \eqref{eq:product}.

Ad \eqref{it:ToeB1}. This follows from the Toeplitz relation $i_m(\31_m\cdot Z)=i_m(\31_m)i_1(Z)$,
since $\31_m\cdot Z=\alpha_m(Z)=Z^m\cdot \31_m$ in $X_m$.

Ad \eqref{it:ToeB2}. Let $p$ and $q$ be distinct primes. Then $p\vee q=pq$.
By Corollary~\ref{lem_prod_of_compacts}, Nica covariance of $\imath$ for the pair
$\theta_{\31_p,\31_p}$ and $\theta_{\31_q,
\31_q}$ is the identity $\imath^{(p)}(\theta_{\31_p, \31_p})\imath^{(q)}(
\theta_{\31_q, \31_q})=\imath^{(pq)}(\theta_{\31_{pq}, \31_{pq}})$. The left hand
side is $w_pw_p^*w_qw_q^*$. The right hand side is $w_{pq}w_{pq}^*$ or, by
commutativity of $\NN^\times$ and relation \eqref{it:ToeB0}, $w_pw_qw_p^*w_q^*$.
Since $w_p, w_q$ are isometries, \eqref{it:ToeB2} follows.

Ad \eqref{it:ToeB3}. Let $p$ be a prime. For $1\leq k<p$ and $\xi\in X_{A,p}$
we have
$$ \theta_{\31_p,\31_p}\theta_{Z^k\31_p,\31_p}(\xi)=\31_p\cdot
    \langle \31_p, Z^k \31_p \rangle_p \langle \31_p,\xi \rangle_p =
(\alpha_p\circ L_p(Z^k))(\alpha_p\circ L_p (\xi))\31_p.
$$
However, this last term is zero by Lemma~\ref{lemma:Em_from_alpha_Lm} (a).
Consequently, $w_p w_p^*u^k w_p w_p^*=i^{(p)}
(\theta_{\31_p,\31_p}\theta_{Z^k\31_p,\31_p})=0$, and relation \eqref{it:ToeB3}
follows.

It remains to prove that $\Tc(X)$ is universal for the relations \eqref{it:ToeB0}--\eqref{it:ToeB3}. It is
clear that $\Tc(X)$ is generated by $u$ and $w_m$, $m\in\NN^\times$. Thus
there is a homomorphism $\rho$
from the universal $C^*$-algebra of the relations \eqref{it:ToeB0}--\eqref{it:ToeB3} onto $\Tc(X)$. We will show
that given a unitary $U$  and isometries
$W_m$, $m\in\NN^\times$ which satisfy the relations
\eqref{it:ToeB0}--\eqref{it:ToeB3}, there
exists a Nica covariant Toeplitz representation $\psi$ of $X$ such that
$\psi_1(Z)=U$ and
$\psi_m(\31_m)=W_m$ for all $m$. Indeed, it is obvious that there exists a $C^*$-algebra
homomorphism $\psi_1$ mapping $Z$ to $U$. Now for $f\in A$ and $m\in \NN^\times$ we define $\psi_m(f\31_m):=
\psi_1(f)W_m$.

To show that $\psi$ is Toeplitz covariant we must verify that
\begin{enumerate}
\item\label{it:UniToeB1} $\psi_m(f\31_m)\psi_n(g\31_n)=\psi_{mn}((f\31_m)\cdot(g\31_n))$,
\item\label{it:UniToeB2} $\psi_m((f\31_m)g)=\psi_m(f\31_m)\psi_1(g)$,
\item\label{it:UniToeB3} $\psi_1(\langle f\31_m, g\31_m \rangle_m) = \psi_m(f\31_m)^*\psi_m(g\31_m)$ for $f,g\in A$.
\end{enumerate}
By continuity (and linearity), it suffices to verify these identities with
$f,g$ replaced by integral powers of $Z$. For example, if $f=Z^k$ and $g=Z^l$ the claim \eqref{it:UniToeB1}
is equivalent to $U^k W_mU^lW_n=U^kU^{ml}W_{mn}$; by \eqref{it:ToeB0} and \eqref{it:ToeB1}, the right-hand
side is $U^kU^{ml}W_mW_n=U^kW_mU^lW_n$.
The claim \eqref{it:UniToeB2} for $f=Z^k$ and
$g=Z^l$ follows from \eqref{it:ToeB1}:
\begin{align*}
\psi_m((Z^k\31_m)\cdot Z^l)
&= \psi_m(Z^k\31_mZ^{lm})=U^{k+lm}W_m\\
&=(U^kW_m)U^l=\psi_m(Z^k\31_m)\psi_1(Z^l).
\end{align*}
To prove equation \eqref{it:UniToeB3} we first compute that
\begin{equation}\label{eq:inner_prod_generators}
\langle Z^k\31_m,Z^l\31_m\rangle_m=L_m(Z^{*k}Z^l)=\begin{cases}
Z^{(l-k)/m}&\text{ if }l\equiv k(\operatorname{mod} m)\\
0&\text{ otherwise}.
\end{cases}
\end{equation}
Thus $\psi_1(\langle Z^k\31_m,Z^l\31_m\rangle_m)$ is $U^{(l-k)/m}$ when
$l\equiv k(\operatorname{mod} m)$ and is
otherwise $0$. On the other hand, $\psi_m(Z^k\31_m)^* \psi_m(Z^l\31_m)=
W_m^*U^{l-k}W_m$, which by \eqref{it:ToeB3}
is zero unless $k\equiv l(\operatorname{mod} m)$, in which case by
\eqref{it:ToeB1} it turns into $U^{(l-k)/m}$, and \eqref{it:UniToeB3}
follows. Note that we have established that $\psi$ is a Toeplitz representation of $X$
only using the relations \eqref{it:ToeB0}, \eqref{it:ToeB1} and \eqref{it:ToeB3}.

It remains to establish that $\psi$ is Nica covariant, since then  the universal property
of $\Tc(X)$ will provide an
inverse for $\rho$. Here the relation \eqref{it:ToeB2} makes its entrance. Let $S\in \Kk(X_{m})$ and
$T\in \Kk(X_{n})$ for $m,n\in \NN^\times$. Since for each $m$ the map
$\psi^{(m)}$ is a homomorphism and  $X_{m}$ is spanned by elements of
the form $Z^k\31_m$, it suffices to verify Nica-covariance for elements of the
form  $S=\theta_{Z^i\31_m, Z^k\31_m}$ and
$T=\theta_{Z^l\31_n, Z^j\31_n}$ for $i,j,k,l\in \ZZ$. If $i=j=0$ then Nica covariance amounts, by Corollary~\ref{lem_prod_of_compacts}, precisely to
\eqref{eq:more_relations} for the elements $U$ and $W_m$.
For arbitrary $i,j$ Nica covariance can be reduced
to this case by applying Remark~\ref{rmk:arbitrary_products_compacts}.
\end{proof}

\begin{rmk}\label{wrelations}

Relations \eqref{it:ToeB0}--\eqref{it:ToeB3} from Theorem~\ref{thm:presentation_T(XA)}
are almost identical with
Laca and Raeburn's relations (T1)--(T5) for $\Tt(\NN\rtimes\NN^\times)$
(see Section
\ref{section:Prelim} above). The only difference between the relations for
$\{u, w_p\mid p\in \Pp\}$  and $\{s, v_p\mid p\in \Pp\}$
lies in the equation $uu^*=1$ making  $u$ a unitary. Indeed, this fact and
\eqref{it:ToeB1} imply
that $w_p=w_puu^*=u^pw_p u^*$, which upon multiplication with $u^*$ from
the left becomes $u^*w_p=u^{p-1}w_pu^*$, i.e. (T4).
\end{rmk}

\begin{cor}\label{LRontoHLS}
There exists a surjective $C^*$-algebra homomorphism from $\Tt(\NN\rtimes\NN^\times)$
onto $\Tc(X)$ sending $s$ to $u$ and $v_p$ to $w_p$ for all $p\in\Pp$.
\end{cor}

The main difficulty in proving Theorem \ref{thm:presentation_T(XA)} consists of dealing with the
Nica covariance. A much simpler argument, already contained in the proof of Theorem
\ref{thm:presentation_T(XA)}, yields the following.

\begin{prop}\label{xbtoeplitz}
The Toeplitz algebra $\Tt(X)$ is the universal $C^*$-algebra generated by a unitary $u$ and isometries
$w_m$, $m\in \NN^\times$, subject to the relations \eqref{it:ToeB0}, \eqref{it:ToeB1} and \eqref{it:ToeB3}.
\end{prop}

In closing of this note, we identify the Cuntz-Pimsner type algebras
associated to the
product system $X$. Since the left action of $A$ in each fiber is by compact
operators, the Cuntz-Pimsner algebra $\Oo(X)$ is, by the definition in
\cite[Proposition 2.9]{Fow2},
equal to the quotient of $\Tt(X)$ by
the ideal generated by all differences $i(f)-i^{(m)}(\phi_m(f))$ for
$f\in A$ and $m\in \NN^\times$. However, the theory of \cite{Fow2} does not
guarantee that $\Oo(X)$ is non-zero. The remedy is to consider the Cuntz-Nica-Pimsner algebra
$\NO(X)$ constructed in \cite{SY}; this algebra is by
definition a quotient of $\Tc(X)$, and due to injectivity of the
left action $\phi_m$ for every $m\in \NN^\times$  contains a copy of the
coefficient algebra $A$, see \cite[Theorem 4.1]{SY}. Moreover, since each
pair of elements in $\NN^\times$ has a least upper bound, and since the left action
$\phi_m$
takes values in the generalized compact operators
$\Kk(X_{m})$, the results of \cite[\S 5.1]{SY} show that $\NO(X)$ and
Fowler's  $\Oo(X)$ are isomorphic.

Denote the images in $\NO(X)$ of generators $u$, $w_m$
under the quotient map from $\Tc(X)$ by the same symbols. Then we have
the following characterization of the Cuntz-Nica-Pimsner algebra of $X$, see also \cite[Theorem 5.2]{BanHLR}.

\begin{prop}\label{prop:cuntz-pimsner-xb}
In addition to the relations from Theorem~\ref{thm:presentation_T(XA)},
the generators of $\NO(X)$ satisfy
\begin{enumerate}\renewcommand{\theenumi}{B\arabic{enumi}}
\setcounter{enumi}{4}
\item\label{it:ToeB4} $\sum_{k=0}^{m-1} u^k w_m w_m^* u^{-k} =1$ for
all $m\in\NN^\times$.
\end{enumerate}
Consequently, $\NO(X)$ is isomorphic to the $C^*$-algebra $\Qq_\NN$ defined
by Cuntz in \cite{Cun1}.
\end{prop}
\begin{proof}
Relation \eqref{it:ToeB4} follows from the fact that for each $m$ we have
\begin{equation}\label{eq:compactidentity}
\phi_m(1)=\sum_{k=0}^{m-1}\theta_{Z^k\31_m, Z^k\31_m}.
\end{equation}
Thus, the generators of $\NO(X)$ satisfy relations (Q1)--(Q3). Whence, by the
universality of $\Qq_\NN$, there exists a $*$-homomorphism from $\Qq_\NN$ to
$\NO(X)$ sending $u$ to $u$ and $s_n$ to $w_n$. This map is clearly surjective and its
injectivity follows from simplicity of $\Qq_\NN$.
\end{proof}


\begin{thebibliography}{00}

\bibitem{BE} G. Boava and R. Exel, \emph{Partial crossed product
description of the $C^*$-algebras associated with integral domains},
arXiv:1010.0967v2[math.OA].

\bibitem{BC} J.-B. Bost and A. Connes,
\emph{ Hecke algebras, type $III$ factors and phase transition with spontaneous symmetry
breaking in number theory,}
Selecta Math. (New Series) {\bf 1} (1995), 411--457.


\bibitem{Bro} N. Brownlowe,
\emph{Realizing the $C^*$-algebra of a higher-rank graph as an Exel crossed product}, J. Operator Theory, in press.

\bibitem{BanHLR} N. Brownlowe, A. an Huef, M. Laca and I. Raeburn, \emph{Boundary quotients of
the Toeplitz algebra of the affine semigroup over the natural numbers}, arXiv:1009.3678v1[math.OA].

\bibitem{CLSV} T. M. Carlsen, N. S. Larsen, A. Sims and S. T. Vittadello,
\emph{ Co-universal algebras associated to product systems, and gauge-invariant uniqueness theorems}, Proc. London Math. Soc., in press.

\bibitem{Cun1} J. Cuntz,
\emph{$C^*$-algebras associated with the $ax+b$-semigroup over $\NN$},
in $K$-Theory and noncommutative geometry (Valladolid, 2006),
European Math. Soc., 2008, pp 201--215.

\bibitem{CunLi} J. Cuntz and X. Li,
\emph{The regular $C^*$-algebra of an integral domain,}  
Quanta of maths, 149–-170, Clay Math. Proc., 11, Amer. Math. Soc., Providence, RI, 2010.


\bibitem{Ex} R. Exel,
\emph{A new look at the crossed product of a $C^*$-algebra by an endomorphism},
Ergodic Theory \& Dynam. Systems {\bf 23} (2003), 1733--1750.


\bibitem{Fow2} N. J. Fowler,
\emph{Discrete product systems of Hilbert bimodules},
Pacific J. Math. {\bf 204} (2002), 335--375.

\bibitem{KLQ} S. Kaliszewski, M. B. Landstad and J. Quigg,
\emph{A crossed product approach to the Cuntz-Li algebras},
arXiv:1012.5285v2[math.OA].


\bibitem{Kw-Le} B. K. Kwa\'{s}niewski and A. V. Lebedev,
{\em Crossed product of a $C^*$-algebra by a semigroup of endomorphisms
generated by partial isometries},
Integral Equations Operator Theory {\bf 63} (2009), 403--425.



\bibitem{Lac-Rae2} M. Laca and I. Raeburn,
\emph{Phase transition on the Toeplitz algebra of the affine semigroup over the natural numbers},
Adv. Math. {\bf 225} (2010), 643--688.


\bibitem{Lar} N. S. Larsen,
\emph{Crossed products by abelian semigroups via transfer operators},
Ergodic Theory \& Dynam. Systems {\bf 30} (2010), 1147--1164.

\bibitem{Lar-Rae} N. S. Larsen and I. Raeburn,
\emph{Projective multi-resolution analyses arising from direct limits of Hilbert modules},
Math. Scand. {\bf 100} (2007), 317--361.

\bibitem{Li} X. Li, \emph{Ring $C^*$-algebras}, Math. Ann. {\bf 348} (2010),
859--898.

\bibitem{Ni} A. Nica,
\emph{$C^*$-algebras generated by isometries and Wiener-Hopf operators},
J. Operator Theory, {\bf 27} (1992), 17-52.

\bibitem{Ph} N. C. Phillips,
\emph{Equivariant $K$-theory and freeness of group actions on $C^*$-algebras},
Lecture Notes in Math. 1274, Springer-Verlag, Berlin, 1987.

\bibitem{P} M. V. Pimsner,
{\emph A class of $C^*$-algebras generalizing both Cuntz-Krieger algebras and
crossed product by ${\mathbb Z}$},
Fields Inst. Commun. {\bf 12} (1997), 189--212.


\bibitem{SY} A. Sims and T. Yeend,
\emph{ $C^*$-algebras associated to product systems of Hilbert bimodules},
J. Operator Theory {\bf 64} (2010), 349--376.

\bibitem{Szy} W. Szyma\'{n}ski,
\emph{The range of $K$-invariants for $C^*$-algebras of infinite graphs,}
Indiana Univ. Math. J. {\bf 51} (2002), 239--249.



\bibitem{Yam} S. Yamashita,
\emph{Cuntz's $ax+b$-semigroup $C^*$-algebra over $\NN$ and product system $C^*$-algebras},
J. Ramanujan Math. Soc. {\bf 24} (2009), 299--322.

\bibitem{Ye} T. Yeend, \emph{Groupoid models for the $C^*$-algebras of topological higher-rank graphs}, J. Operator Theory {\bf 57} (2007), no. 1, 95--120.

\end{thebibliography}
\end{document}